\documentclass[12pt]{amsart}
\usepackage{amssymb}
\usepackage[dvipsnames,usenames]{color}
\usepackage[colorlinks=true,hyperindex]{hyperref}   
\hypersetup{linkcolor=RawSienna, anchorcolor=BurntOrange,
citecolor=OliveGreen, filecolor=BlueViolet, menucolor=OliveGreen,
urlcolor=OliveGreen }

\usepackage[all]{xy}
\UseTips

\renewcommand{\to}[1][]{\xrightarrow{\ #1\ }}

\newcommand{\leftidx}[3]{{\vphantom{#2}}#1#2#3}     
\newcommand{\lstar}[1]{\leftidx{^*}{#1}{}}
\newcommand{\ltightstar}[1]{\leftidx{^*}{\negthinspace#1}{}}
\newcommand{\forget}[1]{}  

\makeatletter
\renewcommand{\theenumi}{\@roman\c@enumi}
\makeatother

\renewcommand{\phi}{\varphi}
\renewcommand{\epsilon}{\varepsilon}
\renewcommand{\theta}{\vartheta}

\newcommand{\sh}{\operatorname{sh}}

\newcommand{\llbracket}{[\negthinspace[}
\newcommand{\rrbracket}{]\negthinspace]}

\def\NN{{\mathbb N}}

\def\AAA{{\mathbb A}}
\def\RR{{\mathbb R}}
\def\QQ{{\mathbb Q}}

\def\cU{\mathcal{U}}

\def\fra{\mathfrak{a}}
\def\frb{\mathfrak{b}}

\def\frm{\mathfrak{m}}
\def\frq{\mathfrak{q}}

 \DeclareMathOperator{\Spec}{Spec}

 \DeclareMathOperator{\fpt}{fpt}

\newtheorem{lemma}{Lemma}[section]
\newtheorem{theorem}[lemma]{Theorem}
\newtheorem{corollary}[lemma]{Corollary}
\newtheorem{proposition}[lemma]{Proposition}

\newtheorem{conjecture}[lemma]{Conjecture}

\theoremstyle{definition}

\newtheorem{remark}[lemma]{Remark}

\theoremstyle{remark}
\newtheorem*{remark*}{Remark}
\newtheorem*{note*}{Note}

\begin{document}

\title{$F$-thresholds of hypersurfaces}

\author[M.~Blickle]{Manuel~Blickle}
\address{Fachbereich Mathematik, Universit\"at  Duisburg-Essen, Standort Essen,
45117 Essen, Germany} \email{{\tt manuel.blickle@uni-essen.de}}

\author[M. Musta\c{t}\v{a}]{Mircea~Musta\c{t}\v{a}}
\address{Department of Mathematics, University of Michigan,
Ann Arbor, MI 48109, USA} \email{{\tt mmustata@umich.edu}}

\author[K.~E.~Smith]{Karen~E.~Smith}
\address{Department of Mathematics, University of Michigan,
Ann Arbor, MI 48109, USA} \email{{\tt kesmith@umich.edu}}

\thanks{2000\,\emph{Mathematics Subject Classification}.
 Primary 13A35; Secondary 14B05.
\newline The second author
  was partially supported by the NSF under grants DMS 0500127 and
   DMS 0111298, and by a Packard Fellowship}
\keywords{$F$-thresholds, test ideals, $F$-modules, non-standard
extension}

\maketitle

\markboth{M.~Blickle, M.~Musta\c t\u a and
K.~E.~Smith}{$F$-thresholds of hypersurfaces}

\section{Introduction}

In characteristic zero one can define invariants of singularities
using \emph{all} divisors over the ambient variety. A key result
that makes these invariants computable says that they can be
determined by the divisors on a resolution of singularities. For
example, if $\fra$ is a sheaf of ideals on a nonsingular variety,
then to every nonnegative real number $\lambda$ one associates the
multiplier ideal ${\mathcal J}(\fra^{\lambda})$. The jumping
exponents of $\fra$ are those $\lambda$ such that ${\mathcal
J}(\fra^{\lambda}) \neq {\mathcal J}(\fra^{\lambda'})$ for every
$\lambda'<\lambda$. It is an easy consequence of the formula giving
the multiplier ideals of $f$ in terms of a log resolution of
singularities, that the jumping exponents form a discrete set of
rational numbers. See for example \cite{Laz}, Ch. 9 for the basic
facts about multiplier ideals and their jumping exponents.

In positive characteristic Hara and Yoshida defined in \cite{HY} an
analogue of the multiplier ideals, the (generalized) test ideals.
The definition works in a very general setting, involving a notion
of tight closure for pairs. In this paper, however, we assume that
we work in a \emph{regular} ring $R$ of characteristic $p>0$ that is
$F$-finite, i.e. such that the Frobenius morphism $F\colon R\to R$
is finite. If $\fra$ is an ideal in $R$ and if $\lambda$ is a
nonnegative real number, then the corresponding test ideal is
denoted by $\tau(\fra^{\lambda})$. In this context we say that
$\lambda$ is an $F$-jumping exponent (or an $F$-threshold) if
$\tau(\fra^{\lambda})\neq\tau(\fra^{\lambda'})$ for every
$\lambda'<\lambda$. The following is our main result about
$F$-jumping exponents in positive characteristic.

\begin{theorem}\label{theorem_introduction}
If $R$ is an $F$-finite regular ring, and if $\fra=(f)$ is a
principal ideal, then the $F$-jumping exponents of $\fra$ are
rational, and they form a discrete set.
\end{theorem}

The discreteness and the rationality of $F$-jumping numbers has been
proved in \cite{BMS} for \emph{every} ideal when the ring $R$ is
essentially of finite type over an $F$-finite field. We mention also
that for $R=k\llbracket x,y\rrbracket$, with $k$ a finite field,
 the above result has been proved in \cite{Ha} using a completely
different approach.

 We stress that the difficulty in attacking this
result does not come from the fact that there is no available
resolution of singularities in positive characteristic. Even in
cases when such a resolution is known to exist, the $F$-jumping
exponents are not simply given in terms of the numerical information
of the resolution. We refer to \cite{MTW} for a discussion of the
known and conjectural connections between the invariants in
characteristic zero and those in characteristic $p$.

In order to prove Theorem~\ref{theorem_introduction} it is enough to
show that the set of $F$-jumping exponents is discrete. The
rationality statement
 follows as in \cite{BMS}: it is
enough to use the fact that if $\lambda$ is an $F$-jumping exponent,
then so are the fractional parts of $p^e\lambda$, for all $e\geq 1$.
Moreover, we will see that it is enough to prove the result in the
case when $R$ is local.

The crucial step in the proof of the theorem relies on showing that
if $\alpha$ is a \emph{rational number}, then $\alpha$ is not an
accumulation point of $F$-jumping exponents of $f$ (irrational
$\alpha$'s are excluded by an inductive argument). The key point in
this step is that (after preparing $\alpha$) we may rephrase the
statement that ``$\alpha$ is not an accumulation point of
$F$-jumping exponents'' as the statement that ``a certain element
$e_\alpha$ of a certain $D_R$--module $M_\alpha$ (which can be
thought of as the $D_R$--module generated by $\frac{1}{f^\alpha}$,
see the paragraph before Lemma \ref{lem2}) is a $D_R$-generator of
$M_\alpha$'' (Corollary \ref{cor1}). Here $D_R$ denotes the ring of
all differential operators of $R$. This $D_R$-module reformulation
is an extension of an argument due to Alvarez-Montaner, Blickle and
Lyubeznik from \cite{AMBL} (one can interpret the main result in
\emph{loc. cit.} as the case $\alpha=1$, when $M_{\alpha}=R_f$).
Since we may assume that $R$ is local, one then finishes the
argument as in \emph{loc. cit.} by using the fact that $M_\alpha$
has finite length as a $D_R$-module (see \cite{Lyu}) to conclude
that $e_\alpha$ indeed generates $M_\alpha$ as a $D_R$-module
(Theorem \ref{thm4}). This argument is carried out in detail in
Section \ref{sec.ratdisc} where also the necessary background and
notation is recalled.

The second half of the paper deals with limits of $F$-pure
thresholds. We apply our rationality result for formal power series
to deduce that every such limit is a rational number. Recall that
the $F$-pure threshold of $\fra$ is the smallest (positive)
$F$-jumping exponent of $\fra$. This invariant has been introduced
by Takagi and Watanabe in \cite{TW} who pointed out the analogy with
the log canonical threshold in characteristic zero. In a fixed
characteristic $p$, we consider the set ${\mathcal T}_n$ consisting
of all $F$-pure thresholds of principal ideals in regular $F$-finite
rings of characteristic $p$ and  dimension $\leq n$. We consider
also the set ${\mathcal T}_n^{\circ}$ of $F$-pure thresholds at the
origin for polynomials $f\in k[x_1,\ldots,x_n]$, where $k$ is an
algebraically closed field of characteristic $p$ (the definition
does not depend on $k$). It is easy to see that every element in
${\mathcal T}_n$ can be computed as the $F$-pure threshold of a
formal power series $f$, and therefore it is the limit of the
$F$-pure thresholds of the various truncations of $f$. Conversely,
we show that every limit of $F$-pure thresholds in bounded dimension
is the $F$-pure threshold of some formal power series.

\begin{theorem}\label{theorem3_introduction}
For every prime $p>0$ and every $n\geq 1$, the set ${\mathcal T}_n$
is the closure of ${\mathcal T}_n^{\circ}$
  In particular, every limit of $F$-pure thresholds of principal ideals
in $F$-finite regular rings of bounded dimension is a rational
number.
\end{theorem}

The proof of Theorem~\ref{theorem3_introduction} uses nonstandard
methods to construct a power series whose $F$-pure threshold is the
limit of a given sequence of $F$-pure thresholds. The necessary
background for Theorem \ref{theorem3_introduction} and its proof are
given in Section \ref{sec.limFthr}. In the final Section
\ref{sec.rem} we record some peculiar features of $F$-pure
thresholds and test ideals, and state some open problems in analogy
with some well-known conjectures in birational geometry.

For an application of non-standard techniques to the study of log
canonical thresholds, see \cite{dFM}. In that case the non-standard
argument is more involved due to the fact that the definition of the
log canonical threshold is "less elementary". While the $F$-pure
threshold is a more subtle invariant than the log canonical
threshold, its definition is "simpler", and this pays off when using
non-standard extensions.

We believe that exploiting the connections and analogies between the
invariants in positive and zero characteristic can be very fruitful.
For example, results on test ideals such as the Subadditivity and
the Restriction Theorems are much easier to prove than for
multiplier ideals, and they imply their characteristic zero
counterpart by reduction mod $p$. Moreover, there are results on
multiplier ideals that so far have been proved only by reduction to
characteristic $p$ (see the work of Takagi \cite{Ta1} and
\cite{Ta2}). On the other hand, certain phenomena that are
well-understood (or just conjectural) in characteristic zero can
point to interesting phenomena in positive characteristic.

\subsection*{Acknowledgements}
The second author is indebted to Caucher Birkar for explaining him
the usefulness of non-standard methods. He would also like to thank
the Institute for Advanced Study, where part of this work has been
carried out.

\section{Discreteness and rationality}\label{sec.ratdisc}

We start by reviewing the definition and some basic properties of
the generalized test ideals from \cite{BMS}. Let $R$ be a regular
ring of characteristic $p>0$. We assume that $R$ is $F$-finite, that
is the Frobenius morphism $F\colon R\to R$, $F(u)=u^p$ is finite.
Note that $F$-finiteness is preserved by taking quotients,
localization and completion (see Example~2.1 in \cite{BMS}).
Moreover, if $R$ is $F$-finite then so are $R[x]$ and $R\llbracket
x\rrbracket$.

For an ideal $J$ and for $e\geq 1$, we put $J^{[p^e]}=(u^{p^e}\mid
u\in J)$. If $\frb$ is an arbitrary ideal in $R$, then we denote by
$\frb^{[1/p^e]}$ the (unique) minimal ideal $J$ such that
$\frb\subseteq J^{[p^e]}$.

Suppose now that $\fra$ is a fixed ideal in $R$ and $\lambda$ is a
positive real number. For every $e\geq 1$ we have
$$\left(\fra^{\lceil\lambda p^e\rceil}\right)^{[1/p^e]}\subseteq
\left(\fra^{\lceil\lambda p^{e+1}\rceil}\right)^{[1/p^{e+1}]},$$
where $\lceil u\rceil$ denotes the smallest integer $\geq u$. This
sequence of ideals stabilizes since $R$ is Noetherian, and the test
ideal is defined as $\tau(\fra^{\lambda}):=\left(\fra^{\lceil\lambda
p^e\rceil}\right)^{[1/p^e]}$ for $e\gg 0$.

Note that if $\lambda>\mu$, then $\tau(\fra^{\lambda})\subseteq
\tau(\fra^{\mu})$. It is shown in \cite{BMS} that for every
$\lambda$ there is $\epsilon>0$ such that
$$\tau(\fra^{\lambda})=\tau(\fra^{\lambda'})$$
for every $\lambda'\in [\lambda,\lambda+\epsilon)$. A positive
$\lambda$ is called an $F$-\emph{jumping exponent} of $\fra$ if
$\tau(\fra^{\lambda})\neq\tau(\fra^{\lambda'})$ for every
$\lambda'<\lambda$. It is convenient to make the convention that $0$
is an $F$-jumping exponent, too.

If $S$ is a multiplicative system in $R$, then
$\tau((S^{-1}\fra)^{\lambda})=S^{-1}\tau(\fra^{\lambda})$.
Similarly, if $R$ is local and $\widehat{R}$ its completion, then
$\tau((\fra\widehat{R})^{\lambda})
=\tau(\fra^{\lambda})\widehat{R}$. In particular, if $\lambda$ is an
$F$-jumping exponent for $S^{-1}\fra$ or for $\fra\widehat{R}$, then
it has to be an $F$-jumping exponent also for $\fra$.

Using the identification of $F$-jumping exponents as $F$-thresholds
one shows in \cite{BMS} that if $\lambda$ is an $F$-jumping
exponent, then so is $p\lambda$. Alternatively, this follows from a
strengthening of the Subadditivity Theorem in this context (see
Proposition~\ref{prop_subadditivity} below and the remark following
it).

From now on we specialize to the case of a principal ideal
$\fra=(f)$. In this case it is shown in \cite{BMS} that for every
$\lambda\geq 1$ we have
$\tau(f^{\lambda})=f\cdot\tau(f^{\lambda-1})$. This implies that if
$\lambda\geq 1$, then $\lambda$ is an $F$-jumping exponent of $f$ if
and only if $\lambda-1$ is such an exponent.

Combining the above two properties, it follows that if $\lambda$ is
an $F$-jumping exponent for $f$, then the fractional parts
$\{p^e\lambda\}$ are also $F$-jumping exponents for all $e\geq 1$.
Hence if we know that the $F$-jumping exponents of $f$ are discrete,
$\lambda$ has to be rational.

\begin{lemma}\label{lem0}
If $\lambda=\frac{m}{p^e}$ for some positive integer $m$, then
$\tau(f^{\lambda})=(f^m)^{[1/p^e]}$.
\end{lemma}

\begin{proof}
By definition, we have
$\tau(f^{\lambda})=\left(f^{mp^{e'-e}}\right)^{[1/p^{e'}]}$ for some
$e'\geq e$. Therefore it is enough to show that for every $g\in R$
and every $\ell\geq 1$ we have
$(g^p)^{[1/p^{\ell+1}]}=g^{[1/p^{\ell}]}$. This in turn follows from
the flatness of the Frobenius morphism: for an ideal $J$, we have
$g\in J^{[p^{\ell}]}$ if and only if $g^p\in J^{[p^{\ell+1}]}$.
\end{proof}

\medskip

We recall now some basic facts about $R[F^e]$-modules and
$D_R$-modules. For details we refer to \cite{Lyu} or \cite{Bli}.
Since $R$ is an $F$-finite regular ring, the ring of differential
operators $D_R\subseteq {\rm End}_{{\mathbb F}_p}(R)$ admits the
following description. For every $e\geq 0$, let $D_R^{e}={\rm
End}_{R^{p^e}}(R)$, hence $D_R^{0}=R$. We have $D_R^{e}\subseteq
D_R^{e+1}$ and
$$D_R=\bigcup_{e\in\NN}D_R^{e}.$$

By definition, $R$ has a canonical structure of left $D_R$-module.
Note also that if $S$ is a multiplicative system in $R$, then we
have a canonical isomorphism $S^{-1}(D_R^{e})\simeq
D_{S^{-1}R}^{e}$. The following lemma is a concrete special case of
so called Frobenius descent (see \cite{AMBL} for a fast
introduction) which states that the Frobenius functor induces an
equivalence of the category of $R$--modules and $D_R^e$--modules. In
this explicit case it shows the relevance of $D_R^{e}$-modules in
our setting.

\begin{lemma}\label{lem1}
The $D_R^{e}$-submodules of $R$ are the ideals of the form
$J^{[p^e]}$ for some ideal $J$. In particular, for every $\frb$, the
ideal $\left(\frb^{[1/p^e]}\right)^{[p^e]}$ is equal to the
$D_R^{e}$-submodule generated by $\frb$.
\end{lemma}

\begin{proof}
By definition, if $P\in D_R^{e}$ and $a,b \in R$, then
$P(a^{p^e}b)=a^{p^e}P(b)$. This implies that every ideal of the form
$J^{[p^e]}$ is a $D_R^{e}$-submodule of $R$.

Conversely, suppose that $I$ is such a submodule, and let $J=\{a\in
R\mid a^{p^e}\in I\}$. We clearly have $J^{[p^e]}\subseteq I$, and
we show that equality holds. If $\frq$ is a prime ideal in $R$, then
$J_{\frq}=\{b\in R_{\frq}\mid b^{p^e}\in I_{\frq}\}$ and
$(J^{[p^e]})_{\frq}=(J_{\frq})^{[p^e]}$. Since $I_{\frq}$ is a
$D_{R_{\frq}}^{e}$-submodule of $R_{\frq}$, it follows that it is
enough to prove that $I=J^{[p^e]}$ when $R$ is local. Hence we may
assume that $R$ is free (and finitely generated) over $R^{p^e}$.

If $u_1,\ldots,u_N$ give a basis of $R$ over $R^{p^e}$, then we get
morphisms $P_i\colon R\to R$ that are $R^{p^e}$-linear by mapping
$u=\sum_{i=1}^Na_i^{p^e}u_i$ to $a_i^{p^e}$. It follows that if
$u\in I$, then $P_i(u)=a_i^{p^e}\in I$ for every $i$, hence $a_i\in
J$ and we have $u\in J^{[p^e]}$.
\end{proof}

\medskip

We denote by $R^e$ the $R$-$R$-bimodule on $R$,  with the left
structure being the usual one, and the right one being induced by
the $e^{\rm th}$ composition of the Frobenius morphism $F^e\colon
R\to R$. We use the scheme-theoretic notation for extension of
scalars via $F^e$: if $M$ is an $R$-module, then we denote by
$F^{e*}M$ the $R$-module $R^e\otimes_RM$. We have a canonical
isomorphism $R^e\otimes_RR^{e'} \simeq R^{e+e'}$ that takes
$a\otimes b$ to $ab^{p^e}$.

The ring $R[F]$ is the noncommutative ring extension of $R$
generated by a variable $F$ such that $Fa=a^pF$ for every $a\in R$.
For every $e\geq 1$ we consider also the subring $R[F^e]\subseteq
R[F]$. An $R[F^e]$-module is hence nothing but an $R$-module $M$
together with an ``action of the $e^{\rm th}$ composition of the
Frobenius on $M$'', that is a group homomorphism $F^e=F^e_M\colon
M\to M$ such that $F^e(au)=a^{p^e}u$ (or more concisely: $F^e_M$ is
an $R$-linear map $M \to F^e_*M$). Due to the adjointness of
$F^{e*}$ and $F^e_*$ this can be rephrased as follows: $M$ is an
$R$-module together with a morphism of left $R$-modules
$\theta_M^e\colon R^{e}\otimes_RM\to M$. The adjointness is
expressed through the equation $\theta_M^{e}(a\otimes u)= aF^e(u)$.

A \emph{unit $R[F^e]$-module} is an $R[F^e]$-module $M$ such that
$\theta^{e}_M$ is an isomorphism. Note that for every $s\geq 1$, the
inclusion $R[F^{se}]\subseteq R[F^e]$ makes any (unit)
$R[F^{e}]$-module into a (unit) $R[F^{se}]$-module. Moreover,
$\theta_{M}^{se}$ can be described recursively as
\[
R^{se}\otimes_RM\simeq
R^{e}\otimes_R(R^{(s-1)e}\otimes_RM)\to[{1\otimes
\theta_M^{(s-1)e}}] R^{e}\otimes_RM\to[{\theta_M^{e}}] M.
\]

Every unit $R[F^e]$-module $M$ has a canonical structure of
$D_R$-module. This is described as follows: since
$D_R=\bigcup_{s\geq 1}D_R^{se}$, it is enough to describe the action
of $P\in D_R^{se}$ on $M$. Using the isomorphism
$\theta_M^{se}\colon R^{se}\otimes_RM\to M$, we let $P$ act by
$P(a\otimes u)=P(a)\otimes u$. A fundamental result of Lyubeznik
\cite{Lyu} says that if $R$ is an algebra of finite type over a
regular local $F$-finite ring, then every finitely generated unit
$R[F^e]$-module has finite length in the category of $D_R$-modules.

It is a general fact that for every $R$-module $P$ and every $e\geq
1$, the pull-back $F^{e*}(M)$  has a natural structure of
$D_R$-module. Moreover, if $P$ is a unit $R[F^e]$-module, then
$\theta_P^e\colon F^{e*}(P)\to P$ is an isomorphism of
$D_R$-modules. For a discussion of this and related facts we refer
to \cite{AMBL}, \S 2.

For simplicity, from now on we assume that $R$ is a domain.  A basic
example of an $R[F]$-module is given by $R_f$, where $f\in R$ is
nonzero. The action of $F$ on $R_f$ is given by the Frobenius
morphism of $R_f$. It is easy to see that $R_f$ is a unit
$R[F]$-module. In fact, we will check this for the following
generalization.

Suppose that $\alpha$ is a positive rational number such that $p$
does not divide the denominator of $\alpha$. Therefore we can find
positive integers $e$ and $r$ such that $\alpha=\frac{r}{p^e-1}$.

We define the $R[F^e]$-module $M_{\alpha}$ as being the $R_f$-free
module with generator $e_{\alpha}$. We think of $e_{\alpha}$
formally as $\frac{1}{f^{\alpha}}$. Since $p^e\alpha=r+\alpha$, this
suggests the following action of $F^e$ on $M_{\alpha}$:
$$F^e\left(\frac{b}{f^m}\cdot
e_{\alpha}\right)=\frac{b^{p^e}}{f^{mp^e+r}}\cdot e_{\alpha}.$$ It
is clear that this makes $M_{\alpha}$ an $R[F^e]$-module.

\begin{lemma}\label{lem2}
For every $\alpha$ as above, $M_{\alpha}$ is a unit $R[F^e]$-module.
\end{lemma}

\begin{proof}
It follows from definition that the morphism $\theta_{M_{\alpha}}^e
\colon R^e\otimes_RM_{\alpha}\to M_{\alpha}$ is given by
$$\theta^e_{M_{\alpha}}\left(a\otimes\frac{b}{f^m}e_{\alpha}\right)=
\frac{ab^{p^e}}{f^{mp^e+r}}e_{\alpha}.$$ It is straightforward to
check that the map $\frac{c}{f^s}e_{\alpha}\to
cf^{s(p^e-1)+r}\otimes\frac{1}{f^s}e_{\alpha}$ is well-defined and
that it is an inverse of $\theta_{M_{\alpha}}^e$.
\end{proof}

\begin{remark}\label{same_structure}
If $e'=es$ for some positive integer $s$, then we may write
\begin{equation}\label{two_ways}
\alpha=\frac{r}{p^e-1}=\frac{r'}{p^{e'}-1},
\end{equation}
with $r'=r\cdot\frac{p^{e'}-1}{p^e-1}$. Since
$$(F^e)^s(e_{\alpha})=\frac{1}{f^{r(1+p^e+\cdots+p^{(s-1)e})}}e_{\alpha}=\frac{1}{f^{r'}}e_{\alpha},$$
we see that the action of $F^{e'}$ on $e_{\alpha}$ is the same for
both ways of writing $\alpha$ in (\ref{two_ways}). In particular,
the $D_R$-module structure on $M_{\alpha}$ depends only on $\alpha$.
\end{remark}

The following lemma relates the module $M_{\alpha}$ to some test
ideals of $f$. If $\alpha=\frac{r}{p^e-1}$ as above and $m\in\NN$,
we put $\alpha_m:=\frac{p^{me}-1}{p^{me}}\cdot\alpha$. Hence the
$\alpha_m$ form a strictly increasing sequence converging to
$\alpha$.

\begin{lemma}\label{lem3}
With the above notation, the following are equivalent:
\begin{enumerate}
\item $\tau(f^{\alpha_m})=\tau(f^{\alpha_{m+1}})$.
\item There is a differential operator $P\in D_R^{(m+1)e}$
such that $P\cdot e_{\alpha}=F^e(e_{\alpha})$.
\end{enumerate}
\end{lemma}

\begin{proof}
It follows from Lemma~\ref{lem0} that we have
$$\tau(f^{\alpha_m})=
\left(f^{r\frac{p^e(p^{me}-1)}{p^e-1}}\right)^{[1/p^{(m+1)e}]},\,{\rm
and}\, \tau(f^{\alpha_{m+1}})=
\left(f^{r\frac{p^{(m+1)e}-1}{p^e-1}}\right)^{[1/p^{(m+1)e}]}.$$
Therefore Lemma~\ref{lem1} implies
$$\tau(f^{\alpha_m})^{[p^{(m+1)e}]}=D_R^{(m+1)e}\cdot
f^{r\frac{p^e(p^{me}-1)}{p^e-1}}\,{\rm and}\,
\tau(f^{\alpha_{m+1}})^{[p^{(m+1)e}]}= D_R^{(m+1)e}\cdot
f^{r\frac{p^{(m+1)e}-1}{p^e-1}}.$$

We always have $\tau(f^{\alpha_{m+1}}) \subseteq\tau(f^{\alpha_m})$
for every $m$. It follows from the above formulas that these ideals
are equal if and only if
 there is $P\in D_R^{(m+1)e}$ such that
\begin{equation}\label{eq_lem3}
f^{r\frac{p^e(p^{me}-1)}{p^e-1}}=P\cdot
f^{r\frac{p^{(m+1)e}-1}{p^e-1}}.
\end{equation}
We claim that this is the case if and only if  $P\cdot e_{\alpha}=
\frac{1}{f^r}e_{\alpha}$ in $M_{\alpha}$. Note first that since
$P\in D_R^{(m+1)e}$, it follows from the description of the action
of $D_R$ on $M_{\alpha}$ that
$$P\cdot
e_{\alpha}=\theta_{M_{\alpha}}^{(m+1)e}(P\otimes
1)(\theta_{M_{\alpha}}^{(m+1)e})^{-1}(e_{\alpha}).
$$
The formula for $(\theta^e_{M_{\alpha}})^{-1}$ in the proof of
Lemma~\ref{lem2} implies that
$$(\theta^{(m+1)e}_{M_{\alpha}})^{-1}(e_{\alpha})=f^{r\frac{p^{(m+1)e}-1}{p^e-1}}\otimes e_{\alpha}$$
and therefore
\[
(\theta_{M_{\alpha}}^{(m+1)e})^{-1}(P\cdot
e_{\alpha})=P\left(f^{r\frac{p^{(m+1)e}-1}{p^e-1}}\right)\otimes
e_{\alpha}=f^{rp^{(m+1)e}}\cdot
P\left(f^{r\frac{p^{(m+1)e}-1}{p^e-1}}\right)\otimes
\frac{1}{f^r}e_{\alpha}.
\]

 On the other hand,
$$(\theta^{(m+1)e}_{M_{\alpha}})^{-1}\left(\frac{1}{f^r}
e_{\alpha}\right)= f^{r\frac{p^e(p^{(m+1)e}-1)}{p^e-1}}\otimes
\frac{1}{f^r}e_{\alpha},$$ hence $P\cdot
e_{\alpha}=\frac{1}{f^r}e_{\alpha}$ if and only if
$$
f^{rp^{(m+1)e}} P\left(f^{r\frac{p^{(m+1)e}-1}{p^e-1}}\right)=
f^{r\frac{p^e(p^{(m+1)e}-1)}{p^e-1}},
$$
which is equivalent to (\ref{eq_lem3}). This completes the proof of
the lemma.
\end{proof}

\begin{remark}\label{intervals}
Since we have $D_R^{(m+1)e}\subseteq D_R^{(m+2)e}$, it follows from
the lemma that if $\tau(f^{\alpha_m})=\tau(f^{\alpha_{m+1}})$, then
$\tau(f^{\alpha_i})=\tau(f^{\alpha_m})$ for every $i\geq m$. In
other words, there is no $F$-jumping exponent in
$(\alpha_m,\alpha)$. See also Proposition~\ref{alternative} below
for an alternative proof of this statement.
\end{remark}

\begin{remark}
Using the language of roots and generators of finitely generated
unit $R[F^e]$-modules as in \cite{Lyu} one can easily show that
$M_\alpha$ is naturally isomorphic to the unit $R[F^e]$-module
generated by
\[
    \beta: R \to[a \mapsto f^r\cdot a \otimes 1] R^e \otimes R \cong F^{e*}R\; .
\]
The unit module generated by $\beta$ is by definition the inductive
limit $\widetilde{M}_{\alpha}$ of the direct system one obtains by
composition of Frobenius powers of the map $\beta$. As an
$R$-module, $\widetilde{M}_{\alpha}$ is isomorphic to $R_f$ but the
action of the Frobenius is not the usual one (except in the case
$r=p^e-1$). One easily checks (by sending the image of $1 \in R$ in
the limit $\widetilde{M}_{\alpha}$ to $e_\alpha \in M_\alpha$) that
$\widetilde{M}_{\alpha}$ and $M_{\alpha}$ are isomorphic as
$R[F^e]$-modules. By construction it follows that $Re_\alpha
\subseteq M_\alpha$ is a root of $M_\alpha$. Therefore
\cite[Corollary 4.4]{AMBL} implies that $e_\alpha$ generates
$M_\alpha$ as a $D_R$-module. In Theorem \ref{thm4} below we will
give a direct proof of this fact.
\end{remark}

\begin{corollary}\label{cor1}
If $\alpha=\frac{r}{p^e-1}$, then $e_{\alpha}$ generates
$M_{\alpha}$ as a $D_R$-module if and only if
 $\alpha$ is not an accumulation point of
$F$-jumping exponents of $f$.
\end{corollary}

\begin{proof}
The $\alpha_m$ form a strictly increasing sequence converging to
$\alpha$, hence $\alpha$ is not an accumulation point of $F$-jumping
exponents if and only if the sequence of ideals
$\{\tau(f^{\alpha_m})\}_m$ stabilizes. By Remark~\ref{intervals},
this is the case if and only if
$\tau(f^{\alpha_m})=\tau(f^{\alpha_{m+1}})$ for some $m$. Since
$D_R=\cup_{m\geq 1}D_R^{me}$, it is clear from Lemma~\ref{lem3} that
if $M_{\alpha} =D_R \cdot e_{\alpha}$, then $\alpha$ is not an
accumulation point of $F$-jumping exponents. Conversely, if this is
the case, then $\frac{1}{f^r}e_{\alpha}\in D_R \cdot e_{\alpha}$. By
Remark~\ref{same_structure}, we see that in fact we have infinitely
many positive integers $r_m$ such that $\frac{1}{f^{r_m}}e_{\alpha}$
lies in $D_R\cdot e_{\alpha}$. Since these elements generate
$M_{\alpha}$ as an $R$-module, we see that $M_{\alpha}=D_R \cdot
e_{\alpha}$.
\end{proof}

\begin{corollary}\label{reduction_to_local}
If $\alpha=\frac{r}{p^e-1}$, then $\alpha$ is not an accumulation
point of $F$-jumping exponents of $f\in R$ if and only if for every
$\frq\in \Spec(R)$, $\alpha$ is not an accumulation point of
$F$-jumping exponents of $\frac{f}{1}\in R_{\frq}$.
\end{corollary}

\begin{proof}
We have seen that $\alpha$ is not an accumulation point of
$F$-jumping exponents of $f$ if and only if $\tau(f^{\alpha_m})
=\tau(f^{\alpha_{m+1}})$ for some $m$. Since taking test ideals
commutes with localization, it is clear that if this property holds
in $R$, then it holds in every $R_{\frq}$. For the converse, note
that if
\begin{equation}\label{case_m}
\tau((fR_{\frq})^{\alpha_m})=\tau((fR_{\frq})^{\alpha_{m+1}}),
\end{equation}
 then
the same holds for all primes $\frq'$ is a neighborhood of $\frq$.
If $U_m$ is the open subset consisting of those $\frq$ for which
(\ref{case_m}) holds, and if ${\rm Spec}(R)=\cup_mU_m$, then
$\Spec(R)=U_{m_0}$ for some $m_0$ (we use the fact that $\Spec(R)$
is quasicompact and that $U_m\subseteq U_{m+1}$ by
Remark~\ref{intervals}). This implies that
$\tau(f^{\alpha_{m_0}})=\tau(f^{\alpha_{m_0+1}})$, hence $\alpha$ is
not an accumulation point of $F$-exponents of $f$.
\end{proof}

\begin{remark}\label{finitely_generated}
It is easy to see that $M_{\alpha}$ is generated by $e_{\alpha}$ as
an $R[F^e]$-module. Indeed, $M_{\alpha}$ is generated as an
$R$-module by the $(F^e)^m(e_{\alpha})=\frac{1}{f^{r
\frac{p^{me}-1}{p^e-1}}}\cdot e_{\alpha}$, with $m\geq 1$.
\end{remark}

\begin{theorem}\label{thm4}
Let $R$ be an $F$-finite regular domain. If $f$ is a nonzero element
in $R$ and $\alpha=\frac{r}{p^e-1}$ for some positive integers $r$
and $e$, then $M_{\alpha}$ is generated over $D_R$ by $e_{\alpha}$.
\end{theorem}

\begin{proof}
Note first that by Corollary~\ref{reduction_to_local}, we may assume
that $R$ is local. Then the argument follows verbatim the argument
for Theorem~4.1 in \cite{AMBL}. Let $N$ denote the $D_R$-submodule
of $M_{\alpha}$ generated by $e_{\alpha}$. Note that we have
$$F^{e*}(N)\subseteq
F^{e*}(M_{\alpha})\to[\theta_{M_{\alpha}}^e] M_{\alpha},$$ hence we
may consider $F^{e*}(N)$ as a submodule of $M_{\alpha}$. We claim
that $N\subseteq F^{e*}(N)$. Since $F^{e*}(N)$ is a $D_R$-submodule
of $M_{\alpha}$, it is enough to show that $e_{\alpha}\in
F^{e*}(N)$. This follows from $e_{\alpha}= f^r\cdot
\theta^e_{M_{\alpha}}(1\otimes e_{\alpha})$.

Theorem~4.3 in \cite{AMBL} shows that this makes $N$ a unit
$R[F^e]$-module, i.e. we have in fact $N=F^{e*}(N)$. The idea is the
following: if $N\neq (F^*)(N)$, then the faithfull flatness of the
Frobenius implies that we have a sequence of \emph{strict}
inclusions
$$N\subsetneq F^{e*}(N)\subsetneq (F^{2e})^*(N)\subsetneq \cdots$$
of $D_R$-submodules of $M_{\alpha}$. This contradicts Lyubeznik's
Theorem \cite{Lyu} saying that as a unit $R[F^e]$-module,
$M_{\alpha}$ has finite length in the category of $D_R$-modules (we
may apply the theorem, since we assume that $R$ is local and
$F$-finite).

Therefore we have $N=(F^{me})^*(N)$ for every $m$. On the other
hand, every element in $M_{\alpha}$ lies in some $(F^{me})^*(N)$.
This follows from
$$\frac{1}{f^{rp^{(m-1)e}}}e_{\alpha}=
\theta_{M_{\alpha}}^{me}(1\otimes e_{\alpha})\in (F^{me})^*(N).$$
Therefore $N=M_{\alpha}$.
\end{proof}

\begin{remark}\label{rem_thm4}
By putting together the above results, we see that under the
hypothesis of Theorem~\ref{thm4}, every rational number $\alpha$
whose denominator is not divisible by $p$ is not an accumulation
point of $F$-jumping exponents of a given $f$. The above proofs
extend to this setting the main result in \cite{AMBL} which deals
with the case $\alpha=1$. In addition, we have dropped the extra
assumption that was imposed in \emph{loc. cit.} in order to apply
Lyubeznik's Theorem.
\end{remark}

Before we proceed to the proof of Theorem~\ref{theorem_introduction}
we show the following lemma which allows us to do induction. Note
that this lemma itself does not require the ideal to be principal.

\begin{lemma}\label{key}
Let $R$ be an $F$-finite regular ring, and $\fra$ an ideal in $R$.
\begin{enumerate}
\item If $\lambda$ is an $F$-jumping exponent of $\fra$, then there is a
prime ideal $\frq$ in $R$ such that $\lambda$ is an $F$-jumping
exponent also of $\fra R_{\frq}$.
\item If $\lambda$ is an \emph{accumulation point} of jumping numbers of $\fra$, then
we can find a \emph{non-maximal} prime ideal $\frq$ such that
$\lambda$ is an $F$-jumping exponent of $\fra R_{\frq}$.
\end{enumerate}
\end{lemma}

\begin{proof}
We may assume that $\lambda>0$, and let us fix
 a strictly increasing sequence of positive numbers $\{\lambda_m\}_m$, with
$\lim_{m\to\infty}\lambda_m=\lambda$. For every $m$ we have
$\tau(\fra^{\lambda})\subseteq\tau(\fra^{\lambda_{m+1}})
\subseteq\tau(\fra^{\lambda_m})$. Let $I_m$ be the ideal
$$(\tau(\fra^{\lambda})\colon \tau(\fra^{\lambda_m}))=\{h\in R\mid
h\cdot\tau(\fra^{\lambda_m})\subseteq\tau(\fra^{\lambda})\}.$$
Therefore $I_m\subseteq I_{m+1}$ for every $m$, and since $R$ is
Noetherian, there is an ideal $I$ such that $I_m=I$ for all $m\gg
0$.

Note that $\lambda$ is an $F$-jumping exponent of $\fra$ if and only
if for every $m$ we have $\tau(\fra^{\lambda}) \neq
\tau(\fra^{\lambda_m})$, or equivalently, $I_m\neq R$. Moreover,
$\lambda$ is an accumulation point of $F$-jumping exponents if and
only if $\tau(\fra^{\lambda_m})\neq\tau(\fra^{\lambda_{m+1}})$ for
every $m$.

If $\lambda$ is an $F$-jumping exponent of $\fra$, let $\frq$ be a
minimal prime containing $I$. Since
$$(\tau((\fra R_{\frq})^{\lambda})\colon \tau((\fra R_{\frq})^{\lambda_m}))=
(\tau(\fra^{\lambda}) R_{\frq}\colon \tau(\fra^{\lambda_m})
R_{\frq})=I_{\frq}\neq R_{\frq},$$ it follows that $\lambda$ is an
$F$-jumping exponent of $\fra R_{\frq}$. This gives i).

We show now that if $\lambda$ is an accumulation point of
$F$-jumping exponents, then we can find $\frq$ as above that is not
a maximal ideal. Equivalently, we need to show that $\dim(R/I)\geq
1$, i.e. $R/I$ is not Artinian. By assumption, if $m\gg 0$ then we
have a strictly decreasing sequence
$$\tau(\fra^{\lambda_m})/\tau(\fra^{\lambda})\supsetneq\tau(\fra^{\lambda_{m+1}})/\tau(\fra^{\lambda})\supsetneq\ldots$$
of finitely generated $R/I$-modules. Therefore $R/I$ cannot be
Artinian, and we get ii).
\end{proof}

\begin{proof}[Proof of Theorems~\ref{theorem_introduction}]
 Since $R$ is a
regular ring, we may write $R=R_1\times\ldots\times R_m$, where
$R_i$ are regular domains. If we write $f=(f_1,\ldots,f_m)$, then
the set of $F$-jumping exponents of $f$ is the union of the sets of
$F$-jumping exponents of each of the $f_i$. Therefore in order to
prove Theorem~\ref{theorem_introduction} for $R$, we may assume that
$R$ is a domain and that $f\neq 0$, the case $f=0$ being trivial.

We have seen that for every $R$ and $f$, the discreteness of the set
of $F$-jumping exponents implies the rationality of every such
exponent. Conversely, if we know that all $F$-jumping exponents are
rational, then they form a discrete set. Indeed, if $\alpha$ is an
accumulation point of $F$-jumping exponents, then $\alpha$ is an
$F$-jumping exponent, too, hence $\alpha\in\QQ$. We can find a
positive integer $m$ such that the denominator of $p^m\alpha$ is not
divisible by $p$. For every $F$-jumping exponent $\beta$, $p^m\beta$
is again an $F$-jumping exponent. Therefore also $p^m\alpha$ is an
accumulation point of $F$-jumping exponents, contradicting
Theorem~\ref{thm4} (see also Remark\ref{rem_thm4}).

We claim that it is enough to prove
Theorem~\ref{theorem_introduction} when $R$ is local. Indeed, given
an arbitrary regular $F$-finite local ring, if $\alpha$ is an
$F$-jumping exponent of $f$, then $\alpha$ is also an $F$-jumping
exponent for $fR_{\frq}$ for some prime $\frq$, by Lemma~\ref{key}.
Hence $\alpha\in\QQ$, by the local case, and as we have seen, this
implies that $R$ satisfies the theorem.

Suppose now that $R$ is local. In particular, $\dim(R)<\infty$, and
we prove the statement by induction on $\dim(R)$. The case
$\dim(R)=0$ is trivial, hence we may assume that the theorem holds
for local rings of dimension $<\dim(R)$. If $\alpha$ is an
accumulation point of $F$-jumping exponents, then Lemma~\ref{key}
implies that there is a prime ideal $\frq$ in $R$, different from
the maximal ideal, such that $\alpha$ is an $F$-jumping exponent for
$fR_{\frq}$. Since $\dim(R_{\frq})<\dim(R)$, we may apply induction
to conclude that $\alpha\in\QQ$. Arguing as before, we get a
contradiction with Theorem~\ref{thm4}. This implies that the set of
$F$-jumping numbers of $f$ is discrete, and as a consequence, it
contains only rational numbers.
\end{proof}

\section{Limits of $F$-pure thresholds}\label{sec.limFthr}

The $F$-pure threshold of an ideal has been introduced in \cite{TW}
and further studied in \cite{MTW}. We start by recalling some basic
properties. In fact, it is convenient to consider more generally the
interpretation of all $F$-jumping exponents of an ideal as
$F$-thresholds, as follows. We refer for details and proofs to
\cite{BMS}.

Suppose that $R$ is an $F$-finite regular ring, $\fra$ is an ideal
in $R$, and $J$ is an ideal such that $\fra\subseteq {\rm Rad}(J)$.
For every $e$ let $\nu(p^e)$ denote the largest $r$ such that
$\fra^r\not\subseteq J^{[p^e]}$ (if there is no such $r$, we put
$\nu(p^e)=0$). By the flatness of the Frobenius one has that
$$\nu(p^e)/p^e\leq\nu(p^{e+1})/p^{e+1}$$
for every $e$. The limit
$$c^J(\fra)=\lim_{e\to\infty}\frac{\nu(p^e)}{p^e}=\sup_{e\geq 1}\frac{\nu(p^e)}{p^e}$$
is finite and it is called the $F$-threshold of $\fra$ with respect
to $J$. It was shown in \cite{BMS}, Corollary 2.24, that the set of
all $F$-thresholds of $\fra$ (when we vary $J$) is equal to the set
of $F$-jumping exponents of $\fra$. More precisely, $c^J(\fra)$ is
the smallest $\lambda$ such that $\tau(\fra^{\lambda})\subseteq J$.
We mention that one can show that for every $e$ we have a strict
inequality $\nu(p^e)/p^e<c^J(\fra)$ (see Proposition~1.7 in
\cite{MTW}).

Note that $\tau((\fra^m)^{\lambda})=\tau(\fra^{m\lambda})$ for every
$\lambda$, hence $c^J(\fra^m)=c^J(\fra)/m$ for every $J$. On the
other hand, we have $c^{J^{[p]}}(\fra)=p\cdot c^J(\fra)$.

\begin{remark}\label{case_principal}
When $\fra=(f)$ is principal, then $f^r\in J^{[p^e]}$ implies
$f^{pr}\in J^{[p^{e+1}]}$. Therefore we have
$$\frac{\nu(p^{e+1})+1}{p^{e+1}}\leq\frac{\nu(p^e)+1}{p^e},$$
hence $c^J(\fra)=\inf_e\frac{\nu(p^e)+1}{p^e}$. It follows that for
every ideal $J$ and every $e$, we have $\nu(p^e)+1=\lceil
c^J(f)p^e\rceil$.
\end{remark}

\smallskip

Suppose now that $R$ is a domain. If $(0)\neq\fra\neq R$, then
$\tau(\fra^0)=R$, and $\tau(\fra^{\lambda})\neq R$ for $\lambda\gg
0$ (in fact, the test ideal is contained in $\fra$ for $\lambda\gg
0$: see, for example, Proposition 2.20 in \cite{BMS}). The $F$-pure
threshold $\fpt(\fra)$ is defined as the smallest positive
$F$-jumping exponent of $\fra$, i.e. it is the smallest $\lambda$
such that $\tau(\fra^{\lambda})\neq R$. We make the convention that
$\fpt(0)=0$ and $\fpt(R)=\infty$.

If $R=R_1\times\cdots\times R_m$ and
$\fra=\fra_1\times\cdots\times\fra_m$, then we define the $F$-pure
threshold of $\fra$ by $\fpt(\fra)=\min_i\fpt(\fra_i)$. We can use
this to reduce the computation of $F$-pure thresholds to the case
when $R$ is a domain, which we will do henceforth.

It is clear from the definition that if $\fra\subseteq\frb$, then
$\fpt(\fra)\leq \fpt(\frb)$. Note also that
$\fpt(\fra^m)=\fpt(\fra)/m$. We record in the following proposition
a few other useful properties of $F$-pure thresholds.

\begin{proposition}\label{F_thresholds}
Let $\fra$ be an ideal in $R$.
\begin{enumerate}
\item If $(R,\frm)$ is local
and $\fra\neq R$, then $\fpt(\fra)=c^{\frm}(\fra)$.
\item If $S$ is a multiplicative system in $R$, then
$\fpt(S^{-1}\fra)\geq \fpt(\fra)$.
\item If $\frm$ is a maximal ideal containing $\fra$, then $\fpt(\fra
R_{\frm})=c^{\frm}(\fra)$.
\item We have $\fpt(\fra)=\min_{\frq}\fpt(\fra R_{\frq})$, where
the minimum is either over the prime ideals, or over the maximal
ideals of $R$.
\item If $R$ is a local ring and $\widehat{R}$ is its completion, then $\fpt(\fra\widehat{R})=\fpt(\fra)$.
\end{enumerate}
\end{proposition}

\begin{proof}
For all assertions we may assume that $R$ is a domain and that
$\fra$ is a proper nonzero ideal. (i) is clear, and (ii) follows
from the fact taking test ideals commutes with localization. For
(iii), note that since $\frm^{[p^e]}$ is $\frm$-primary, we get
$(\frm R_{\frm})^{[p^e]}\cap R=\frm^{[p^e]}$, which gives our
statement.

In order to prove (iv), let $c=\fpt(\fra)$. If $I=\tau(\fra^c)$,
then for every prime ideal $\frq$ containing $I$ we have $\tau((\fra
R_{\frq})^c)=I_{\frq}\neq R_{\frq}$, hence $c=\fpt(\fra R_{\frq})$.
(v) follows from the fact that taking test ideals commutes with
completion.
\end{proof}

The following lemma will allow us to approximate arbitrary $F$-pure
thresholds by $F$-pure thresholds of polynomials. The statement can
be found in \cite{MTW}, but we give the proof for completeness.

\begin{lemma}
If $J$ is an ideal in $R$, and $\fra$, $\frb$ are ideals contained
in ${\rm Rad}(J)$, then $c^J(\fra+\frb)\leq c^J(\fra)+c^J(\frb)$.
\end{lemma}

\begin{proof}
If $e\geq 1$ and $r$, $s$ are such that $\fra^r\subseteq J^{[p^e]}$
and $\frb^s\subseteq J^{[p^e]}$, then it is clear that
$(\fra+\frb)^{r+s}\subseteq J^{[p^e]}$. The assertion of the lemma
follows now from the definition of $F$-thresholds.
\end{proof}

\begin{corollary}\label{cor3}
If $\frm^{[p^s]}\subseteq J\subseteq\frm$ for some maximal ideal
$\frm$ and some $s\geq 1$, and if $f$, $g\in \frm$ are such that
$f-g\in\frm^N$, then
$$|c^{J}(f)-c^J(g)|\leq \frac{p^s\cdot \dim(R_{\frm})}{N}.$$
\end{corollary}

\begin{proof}
Since $f-g\in \frm^N$, the lemma gives
$$|c^J(f)-c^J(g)|\leq c^J(\frm^N)\leq c^{\frm^{[p^s]}}(\frm^N)=\frac{p^s\cdot
\dim(R_{\frm})}{N},$$ where we use the fact that
$c^{\frm}(\frm)=\dim(R_{\frm})$.
\end{proof}

\smallskip

We turn now to the study of the set of $F$-pure thresholds of
principal ideals in bounded dimension. Recall that since
$\tau(f)=(f)$, we have $\fpt(f)\leq 1$ for every non-invertible $f$
in a domain $R$. If $k$ is a field and $f\in k[x_1,\ldots,x_n]$ is
such that $f(0)=0$, then we denote by $\fpt_0(f)$ the $F$-pure
threshold of the image of $f$ in
$k[x_1,\ldots,x_n]_{(x_1,\ldots,x_n)}$.

We fix now the characteristic $p$. For $n\geq 1$ and for every field
$k$ of characteristic $p$ we denote by ${\mathcal T}^{\circ}_n(k)$
the set of $F$-pure thresholds $\fpt_0(f)$, where $f\in
k[x_1,\ldots,x_n]$ is such that $f(0)=0$. We also put ${\mathcal
T}_n$ for the set of $F$-pure thresholds $\fpt(g)$, where $g\in R$
is not invertible and $R$ is an $F$-finite regular domain of
dimension $\leq n$.

\begin{theorem}\label{thm5}
If the characteristic $p$ is fixed, then
\begin{enumerate}
\item For every $n$ we have ${\mathcal T}_n^{\circ}(k)\subseteq
{\mathcal T}^{\circ}_{n+1}(k)$.
\item For every field extension $K/k$, we have
${\mathcal T}^{\circ}_n(k)\subseteq {\mathcal T}^{\circ}_n(K)$.
\item The set ${\mathcal T}^{\circ}_n(k)$ does not depend on $k$
if $k$ is algebraically closed (from now on we simply denote this
set by ${\mathcal T}_n^{\circ}$).
\item The set ${\mathcal T}_n$ is contained in the closure of
${\mathcal T}_n^{\circ}$.
\end{enumerate}
\end{theorem}

\begin{proof}
If $f\in k[x_1,\ldots,x_n]$, then
$\fpt_0(f)=\lim_{e\to\infty}\nu(p^e)/p^e$, where $\nu(p^e)$ is the
largest $r$ such that $f^r\not\in (x_1^{p^e},\ldots,x_n^{p^e})$. It
is clear that $\nu(p^e)$ does not change if we consider instead $f$
in $k[x_1,\ldots,x_n]$ or in $K[x_1,\ldots,x_n]$, for some field
extension $K$ of $k$. This gives the first two assertions.

 For (iii), it is enough to show that for every $d\geq 1$ and
$c>0$, the existence of $f\in k[x_1,\ldots,x_n]_{\leq d}$ (these are
the polynomials of degree $\leq d$) with $\fpt_0(f)=c$ does not
depend on the algebraically closed field $k$. Note that
$$\{f\in k[x_1,\ldots,x_n]_{\leq d}\mid f(0)=0\} \cong k^N,$$
 can be viewed as the set of $k$-valued points of a suitable
 affine space $\AAA^N$. Recall first from \cite[Proposition~3.6]{BMS} that
the denominators of the $F$-jumping numbers of $f$ in the above set
are bounded in terms of $d$, $n$ and the characteristic $p$; in
particular, the bound is independent of the field $k$. Therefore we
have $0<a_1<\ldots<a_{m(d,n,p)}\leq 1$ such that for every nonzero
$f$ of degree $\leq d$ and with $f(0)=0$, we have $\fpt_0(f)=a_i$
for some $i$. From now on we may assume that $c=a_j$ for some $j$.

On the other hand, for every $a>0$, the condition that
$\fpt_0(f)\leq a$ is equivalent with $f^{\lfloor ap^e\rfloor+1}\in
(x_1^{p^e},\ldots,x_n^{p^e})$ for every $e$. Here we have denoted by
$\lfloor u\rfloor$ the largest integer $\leq u$. Since the degree of
$f$ is bounded above by $d$, it follows that the condition for $f$
to have $\fpt_0(f)\geq a$ is an intersection of closed conditions
defined over ${\mathbb F}_p$, hence it is closed and defined over
${\mathbb F}_p$. In other words, there is a closed subscheme
$Z_a\subseteq \AAA^N_{{\mathbb F}_p}$ such that $f\in
k[x_1,\ldots,x_n]_{\leq d}$ with $f(0)=0$ has $\fpt_0(f)\leq a$ if
and only if $f$ corresponds to a $k$-valued point of $Z_a$.

Given $c=a_j$, let us choose $a$ and $b$ such that
$a_{j-1}<a<a_j<b<a_{j+1}$. It follows that  there is $f\in
k[x_1,\ldots,x_n]_{\leq d}$ with $\fpt_0(f)=a_j$ if and only if
$Z_a(k)\neq Z_b(k)$. This condition does not depend on $k$ if $k$ is
algebraically closed, which proves (iii).

In order to prove (iv), consider $c\in {\mathcal T}_n$. By
definition, there is an $F$-finite regular ring $R$ of dimension
$\leq n$ and $g\in R$ non-invertible such that $c=\fpt(g)$. By (iv)
and (v) in Proposition~\ref{F_thresholds}, we may assume that $R$ is
local and complete. Since $R$ is regular and contains a field,
Cohen's Structure Theorem implies that $R\simeq k\llbracket
x_1,\ldots,x_m\rrbracket$ for some field $k$ and some $m\leq n$.
Since $k$ is a quotient of $R$, it follows that $[k\colon
k^p]<\infty$. Using again Proposition~\ref{F_thresholds} we see that
if $f\in k[x_1,\ldots,x_m]$ with $f(0)$, then $\fpt_0(f)$ is equal
to its $F$-pure threshold in $R$. If we denote by $g_{\leq d}$ the
truncation of $g$ up to degree $\leq d$, it follows from
Corollary~\ref{cor3} that
$$|\fpt_0(g_{\leq d})-\fpt(g)|\leq\frac{m}{d+1}.$$
This implies that $c=\fpt(g)$ is in the closure of ${\mathcal
T}^{\circ}_n$.
\end{proof}

\smallskip

Our next goal is to prove Theorem~\ref{theorem3_introduction}. In
order to show that ${\mathcal T}_n$ is closed, the key idea is to
associate to a sequence of polynomials $f_m\in k[x_1,\ldots,x_n]$
with $\lim_{m\to\infty}\fpt_0(f_m)=\alpha$ a formal power series $f$
over some extension field of $k$ such that $\fpt(f)=\alpha$. This
will be done using some basic constructions from non-standard
analysis.  We briefly present these constructions, and refer for
details to \cite{Gol}.

Recall that an \emph{ultrafilter} $\cU$ on the set $\NN$ is a
collection of subsets of $\NN$ with the following properties:
\begin{enumerate}
\item $\emptyset\not\in \cU$.
\item If $A\in {\mathcal U}$ and $B\supseteq A$, then $B\in
\cU$.
\item If $A$, $B\in\cU$, then $A\cap B\in\cU$.
\item If $A\subseteq \NN$, then either $A$ or
$\NN\smallsetminus A$ is in $\cU$.
\end{enumerate}

An ultrafilter ${\mathcal U}$ is \emph{non-principal} if every
subset of $\NN$ whose complement is finite belongs to $\cU$. It
follows from Zorn's Lemma that there are non-principal ultrafilters
on $\NN$, and from now on we fix one such non-principal ultrafilter
$\cU$. Since we will need this later we point out that the
properties (iii) and (iv) easily imply that if $A = U_1\cup \ldots
\cup U_n$ is a finite disjoint union, and if $A\in \cU$, then one
and only one of the $U_i$ is in $\cU$.

If $\{A_m\}_{m\in\NN}$ is a sequence of sets, then one defines on
the product $\prod_{m\in\NN}A_m$ the equivalence relation
$$(a_m)_m\sim (b_m)_m\,\,{\rm iff}\,\,\{m\mid a_m=b_m\}\in\cU.$$
The set of equivalence classes is called the \emph{ultraproduct}
(with respect to the ultrafilter $\cU$) and we denote it here by
$[A_m]$ and the class of $(a_m)_m$ is denoted by $[a_m]$. Similarly,
from a sequence of functions $f_m\colon A_m \to B_m$ we get a
function $[f_m]\colon [A_m] \to[] [B_m]$ that takes $[a_m]$ to
$[f(a_m)]$.

When $A_m=A$ for every $m$, the corresponding ultraproduct is
denoted by $\ltightstar{A}$ and it is called the non-standard
extension of $A$. Note that we have an injective map
$A\hookrightarrow \ltightstar{A}$ that takes $a$ to the class of
$(a,a,\ldots)$. Similarly, a function $u\colon A\to B$ has a
non-standard extension $\ltightstar{u}\colon \ltightstar{A}\to
\ltightstar{B}$. As a general principle one observes that if $A$ has
an algebraic structure, then $\ltightstar{A}$ has a similar
structure, too. For example, $\lstar{\RR}$ is an ordered field, and
if $k$ is an algebraically closed field, then so is
$\ltightstar{k}$.

If we have a sequence of polynomials $f_m\in k[x_1,\ldots,x_n]$,
then we obtain an \emph{internal hyperpolynomial} $F=[f_m]\in
\ltightstar{(k[x_1,\ldots,x_n])}$. We can view any polynomial $g \in
k[x_1,\ldots,x_n]$ (or more generally any power series) as a
function $\NN^n \to k$ given by sending the tuple $(m_1,\ldots,m_n)$
to the coefficient of the monomial $x_1^{m_1}\cdot\ldots\cdot
x_n^{m_n}$ in $g$. Hence we can view $F$ as a function
$(\lstar{\NN})^n\to \ltightstar{k}$. If we restrict this function
$F$ to $\NN^n$, then we get a formal power series $f\in
(\ltightstar{k})\llbracket x_1,\ldots,x_n \rrbracket$. Hence we have
the following natural maps
\[
    {k[x_1,\ldots,x_n]}  \hookrightarrow {\ltightstar{(k[x_1,\ldots,x_n])}}  \to {(\ltightstar{k})\llbracket x_1,\ldots,x_n\rrbracket}
\]
Note that if $f_m(0)=0$ for every $m$, then $f$ lies in the maximal
ideal, i.e. $f(0)=0$. After these preparations, we can prove that
${\mathcal T}_n$ is closed.

\begin{proof}[Proof of Theorem~\ref{theorem3_introduction}]
We fix an algebraically closed field $k$ of characteristic $p$. We
have already seen in Theorem~\ref{thm5} that ${\mathcal
T}_n^{\circ}={\mathcal T}_n^{\circ}(k)$ is independent on the choice
of $k$, and that it is dense in ${\mathcal T}_n$. Therefore, in
order to show that ${\mathcal T}_n$ is closed, it is enough to show
that if we have a sequence $f_m\in k[x_1,\ldots,x_n]$ with
$f_m(0)=0$ and $\lim_{m\to\infty}\fpt_0(f_m)=\alpha$, then
$\alpha\in {\mathcal T_n}$. In fact, we will show that if $f\in
(\ltightstar{k})\llbracket x_1,\ldots,x_n\rrbracket$ is the formal
power series associated to $F=[f_m]$ as above, then
$\alpha=\fpt(f)$. The function $\fpt_0\colon \{g\in
k[x_1,\ldots,x_n]\mid g(0)=0\}\to\RR$ extends to
$$\ltightstar{\fpt_0}\colon\{F\in
\ltightstar{(k[x_1,\ldots,x_n])}\mid F(0)=0\}\to\lstar{\RR}$$ such
that $\ltightstar{\fpt_0}(F)=[\fpt_0(f_m)]\in \lstar{\RR}$.

It is a general fact that for every element $w\in \lstar{\RR}$,
there is a unique real number, its \emph{shadow} denoted by ${\rm
sh}(w)$ such that $|w-{\rm sh}(w)|<\epsilon$ for every positive real
number $\epsilon$ (see \cite{Gol}, \S 5.6). Moreover,  Theorem~6.1
in \emph{loc. cit}.  implies that if $c_m$ is a sequence of real
numbers converging to $c$, then ${\rm sh}([c_m])=c$. Hence in order
to conclude we need to show that ${\rm
sh}(\ltightstar{\fpt_0(F)})=\fpt(f)$. Denoting by
$(\underline{\phantom{m}})^\circ$ the subsets consisting of
polynomials or power series vanishing at zero, this means that we
have to show that the following diagram commutes
\[
\xymatrix{
    {\ltightstar{(k[x_1,\ldots,x_n]^{\circ})}} \ar[r]\ar[d]^{\ltightstar{\fpt_0}}
    &{(\ltightstar{k})\llbracket x_1,\ldots,x_n\rrbracket^{\circ}}\ar[d]^{\fpt_0} \\
    {\lstar{\RR}} \ar[r]^{\sh} & {\RR}
    }
\]
For every positive integer $d$, we denote by $f_{\leq d}$ the
truncation of $f$ of degree $\leq d$. It follows from
Corollary~\ref{cor3} that
\begin{equation}\label{eq101}
|\fpt(f)-\fpt(f_{\leq d})|\leq \frac{n}{d+1}.
\end{equation}
We have by definition $f_{\leq d}=[(f_m)_{\leq d}]$, hence
\begin{equation}\label{eq102}
|\ltightstar{\fpt_0}(F)-\ltightstar{\fpt_0}(f_{\leq
d})|=[|\fpt_0(f_m)-\fpt_0((f_m)_{\leq d})|] \leq n/(d+1).
\end{equation}
 If we show that $\fpt(f_{\leq d})=
\ltightstar{\fpt_0}(f_{\leq d})$, then we are done. Indeed, we
deduce from (\ref{eq101}) and (\ref{eq102}) that
$|\fpt(f)-\ltightstar{\fpt_0}(F)|\leq 2n/(d+1)$ for any $d$. Since
$\fpt(f)\in\RR$, this implies $\fpt(f)={\rm
sh}(\ltightstar{\fpt_0(F)})$.

To simplify the notation we put $g_m=(f_m)_{\leq d}$ and $g=[g_m]$.
It follows from Proposition~\ref{F_thresholds} that for any
polynomial $h\in k[x_1,\ldots,x_n]$ with $h(0)=0$, we may compute
$\fpt_0(h)$ by considering $h$ in $(k^*)\llbracket
x_1,\ldots,x_n\rrbracket$.

Recall that if we bound the degree of a polynomial, then we can also
bound the denominators of its $F$-jumping exponents, independently
of the base field (see Proposition~3.6 in \cite{BMS}). Since the
$F$-pure thresholds of principal ideals are bounded above by $1$, it
follows that there is a finite set of rational numbers $A$ such that
$\fpt(g)\in A$ and $\fpt(g_m)\in A$ for every $m$. This implies that
there is a unique $a\in A$ such that $\{m\in\NN\mid
\fpt(g_m)=a\}\in\cU$.

Let us prove for example that $a\leq \fpt(g)$ (the reverse
inequality follows by an analogous argument). We choose positive
integers $r$ and $e$ such that $a\geq\frac{r+1}{p^e}$ and every
element in $A$ that is $<a$ is also $<\frac{r}{p^e}$. Since
$\fpt(g)\in A$, if we show that $g^r\not\in
(x_1^{p^e},\ldots,x_n^{p^e})$, then $\fpt(g)\geq\frac{r}{p^e}$,
hence $\fpt(g)\geq a$. Note that if for some $m$ we have
$\fpt(g_m)=a$, then using Remark~\ref{case_principal} we get
$g_m^r\not\in (x_1^{p^e},\ldots,x_n^{p^e})$, hence there is a
monomial $x_1^{b_1}\cdots x_n^{b_n}$ with all $b_i\leq p^e-1$ that
does not appear in $g_m$. Note that the set of those $m$ that
satisfy this condition is in $\cU$. Since there are only finitely
many monomials as above, it follows that after possibly passing to a
smaller subset we may assume in addition that the same monomial
works for all these $m$. This means that the coefficient of the
monomial in $g^r=[g_m^r]$ is nonzero, hence $g^r\not\in
(x_1^{p^e},\ldots,x_n^{p^e})$, as required. This proves that
${\mathcal T}_n$ is closed.

The last assertion in the theorem follows since ${\mathcal T}_n$ is
contained in $\QQ$ by Theorem~\ref{theorem_introduction}.

\end{proof}

\section{Remarks and open problems}\label{sec.rem}

Recall that test ideals satisfy the following analogue of the
Subadditivity Theorem for multiplier ideals. If $\fra$ and $\frb$
are ideals in $R$ and if $\lambda\in\RR_+$, then
$$\tau((\fra\frb)^{\lambda})\subseteq\tau(\fra^{\lambda})\cdot\tau(\frb^{\lambda}).$$
See, for example, Lemma~2.10 in \cite{BMS} for a proof. In the case
of a $p$-power, we have the following strengthening.

\begin{proposition}\label{prop_subadditivity}
If $\fra$ is an ideal in $\RR_+$ and if $\lambda\in\RR_+$, then
\begin{equation}\label{eq_subadditivity}
\tau(\fra^{p\lambda})\subseteq \tau(\fra^{\lambda})^{[p]}.
\end{equation}
Moreover, if $\tau(\fra^{p\lambda})\subseteq J^{[p]}$ for some ideal
$J$, then $\tau(\fra^{\lambda})\subseteq J$.
\end{proposition}

\begin{proof}
If $e\gg 0$, then $\tau(\fra^{\lambda})=\left(\fra^{\lceil\lambda
p^{e+1}\rceil}\right)^{[1/p^{e+1}]}$ and
$\tau(\fra^{p\lambda})=\left(\fra^{\lceil\lambda
p^{e+1}\rceil}\right)^{[1/p^e]}$. By definition we have
$$\fra^{\lceil\lambda
p^{e+1}\rceil}\subseteq \left(\left((\fra^{\lceil\lambda
p^{e+1}\rceil})^{[1/p^{e+1}]}\right)^{[p]}\right)^{[p^e]},$$ hence
$\left(\fra^{\lceil\lambda p^{e+1}\rceil}\right)^{[1/p^e]}\subseteq
\left(\left(\fra^{\lceil\lambda
p^{e+1}\rceil}\right)^{[1/p^{e+1}]}\right)^{[p]}$, which gives
(\ref{eq_subadditivity}).

Suppose now that
$$\tau(\fra^{p\lambda})=\left(\fra^{\lceil \lambda p^{e+1}\rceil}\right)^{[1/p^e]} \subseteq
J^{[p]}.$$ It follows that $\fra^{\lceil\lambda
p^{e+1}\rceil}\subseteq (J^{[p]})^{[p^e]}=J^{[p^{e+1}]}$. Therefore
$\tau(\fra^{\lambda})=\left(\fra^{\lceil\lambda
p^{e+1}\rceil}\right)^{[1/p^{e+1}]}\subseteq J$.
\end{proof}

\begin{remark}
The above proposition gives another proof for the fact that if
$\lambda$ is an $F$-jumping exponent for an ideal $\fra$, then also
$p\lambda$ is an $F$-jumping exponent. More precisely, if
$\epsilon>0$ is such that
$\tau(\fra^{p\lambda-\epsilon})=\tau(\fra^{p\lambda})$, then
$\tau(\fra^{\lambda-\frac{\epsilon}{p}})=\tau(\fra^{\lambda})$.
\end{remark}

\smallskip

As we have already mentioned, there are many analogies between the
$F$-pure threshold and a characteristic zero invariant that is very
much studied, the log canonical threshold (see \cite{TW} and also
\cite{MTW}). However, in characteristic zero there is not much
difference in considering log canonical thresholds of principal or
of arbitrary ideals. This is not the case in characteristic $p$. For
example, every rational number $c$ is equal to $\fpt(\fra)$ for some
ideal $\fra$ in some polynomial ring $R$: if $c=\frac{n}{r}$, we may
take $R=k[x_1,\ldots,x_n]$ and $\fra= (x_1,\ldots,x_n)^r$. On the
other hand, as the following proposition shows, there are intervals
in $(0,1)$ containing no $F$-pure threshold of a principal ideal
\emph{in any dimension}. For example, there is no such $F$-pure
threshold in $\left(1-\frac{1}{p},1\right)$. The proposition follows
also from the results in Section \ref{sec.ratdisc} (see
Remark~\ref{intervals}), but we give below a direct argument.

\begin{proposition}\label{alternative}
Let $R$ be a regular $F$-finite ring of characteristic $p$, and
$f\in R$.
\begin{enumerate}
\item Let $\alpha=\frac{r}{p^e-1}$ for some positive $r$ and
$e$, and we put $\alpha_m=\left(1-\frac{1}{p^{me}}\right)\alpha$ for
$m\geq 0$. If there is an $F$-jumping exponent of $f$ in
$(\alpha_{m+1},\alpha_{m+2}]$, then there is an $F$-jumping exponent
of $f$ also in $(\alpha_m,\alpha_{m+1}]$.
\item For every $e\geq 1$ and every $0\leq a\leq p^e-1$, the
$F$-pure threshold $\fpt(f)$ does not lie in
$\left(\frac{a}{p^e},\frac{a}{p^e-1}\right)$.
\end{enumerate}
\end{proposition}

Note that (ii) gives for every $e$ open intervals of total length
$\sum_{0\leq i\leq p^e-1} i/p^e(p^e-1)=1/2$ containing no $F$-pure
threshold of a principal ideal in characteristic $p$. One should
compare this with the characteristic zero case, when every $c\in
(0,1]$ is the log canonical threshold of some hypersurface: for
example, if $c=\frac{n}{r}\leq 1$, then $c$ is the log canonical
threshold of $\sum_{i=1}^nx_i^r$.

\begin{proof}[Proof of Proposition~\ref{alternative}]
The assertion in (i) follows from the fact that if $\lambda\in
(\alpha_{m+1},\alpha_{m+2}]$, then $p^e\lambda-r\in
(\alpha_m,\alpha_{m+1}]$, and we have seen that $p^e\lambda-r$ is an
$F$-jumping exponent if $\lambda$ is. In particular, we see that if
$\lambda$ is an $F$-jumping exponent in
$\left(\alpha_1,\alpha\right)$, then there is another positive
$F$-jumping exponent $<\lambda$. Hence $\lambda$ is not the $F$-pure
threshold of $f$.
\end{proof}

\smallskip

Motivated by the analogy with some important conjectures on log
canonical thresholds in characteristic zero (see \cite{kollar}, \S
8) we make the following conjectures on $F$-pure thresholds.

\begin{conjecture}\label{ACC}
For every prime $p$ and every $n$, the set ${\mathcal T}_n$
satisfies ACC (the Ascending Chain Condition), i.e. it contains no
strictly increasing sequences.
\end{conjecture}

\begin{remark}\label{truncation}
Note that if $f\in k\llbracket x_1,\ldots,x_n\rrbracket$ lies in the
maximal ideal, then Corollary~\ref{cor3} implies that whenever
$f-g\in (x_1,\ldots,x_n)^d$, we have $|\fpt(f)-\fpt(g)|\leq n/d$.
Therefore the above conjecture predicts that given $f$, there is $d$
such that $\fpt(f)\geq \fpt(f+h)$ for all $h\in (x_1,\ldots,x_n)^d$.
However, even this special case is not known.
\end{remark}

\begin{conjecture}\label{limit1}
For every prime $p$ and every $n\geq 1$, the accumulation points of
${\mathcal T}_n$ are contained in ${\mathcal T}_{n-1}$ ${\rm (}$by
convention, ${\mathcal T}_0=\{0\}$${\rm )}$.
\end{conjecture}

We have seen in the previous section that the set ${\mathcal T}_n$
is the closure of ${\mathcal T}_n^{\circ}$. In fact, we make also
the following conjecture.

\begin{conjecture}\label{conj3}
For every prime $p$ and every $n$, every $F$-pure threshold in
dimension $\leq n$ can be obtained as the $F$-pure threshold at the
origin of some polynomial in $k[x_1,\ldots,x_n]$, i.e.
 we have ${\mathcal
T}_n={\mathcal T}_n^{\circ}$.
\end{conjecture}

\providecommand{\bysame}{\leavevmode \hbox \o3em
{\hrulefill}\thinspace}


\end{document}